\documentclass[english]{article}
\usepackage[T1]{fontenc}
\usepackage[latin9]{inputenc}
\usepackage{xcolor}
\usepackage{mathtools}
\usepackage{amsmath}
\usepackage{amsthm}
\usepackage{amssymb}

\PassOptionsToPackage{normalem}{ulem}
\usepackage{ulem}

\usepackage{xcolor}
\usepackage{hyperref}
\hypersetup{
    colorlinks,
    linkcolor={blue!90!black},
    citecolor={blue!90!black},
    urlcolor={blue!90!black}
}

\makeatletter

\providecolor{lyxadded}{rgb}{0,0,1}
\providecolor{lyxdeleted}{rgb}{1,0,0}

\DeclareRobustCommand{\lyxsout}[1]{\ifx\\#1\else\sout{#1}\fi}

\theoremstyle{plain}
\newtheorem{thm}{\protect\theoremname}[section]
\theoremstyle{definition}
\newtheorem{defn}[thm]{\protect\definitionname}
\theoremstyle{remark}
\newtheorem{rem}[thm]{\protect\remarkname}
\theoremstyle{plain}
\newtheorem{prop}[thm]{\protect\propositionname}
\theoremstyle{plain}
\newtheorem{lem}[thm]{\protect\lemmaname}
\theoremstyle{plain}
\newtheorem{conjecture}[thm]{\protect\conjecturename}

\usepackage{babel}

\providecommand{\definitionname}{Definition}
\providecommand{\lemmaname}{Lemma}
\providecommand{\propositionname}{Proposition}
\providecommand{\remarkname}{Remark}
\providecommand{\theoremname}{Theorem}

\makeatother

\usepackage{babel}
\providecommand{\conjecturename}{Conjecture}
\providecommand{\definitionname}{Definition}
\providecommand{\lemmaname}{Lemma}
\providecommand{\propositionname}{Proposition}
\providecommand{\remarkname}{Remark}
\providecommand{\theoremname}{Theorem}

\begin{document}
\title{Spectral gap and embedded trees for the Laplacian of the Erd\H{o}s-R\'enyi graph}
\author{Raphael Ducatez\footnote{Universit\'e Claude-Bernard Lyon 1}\, and  
Renaud Rivier\footnote{Universit\'e de Gen\`eve}}
\maketitle
\begin{abstract}
For the Erd\H{o}s-R\'enyi  graph of size $N$ with mean degree $(1+o(1))\frac{\log N}{t+1}\leq d\leq(1-o(1))\frac{\log N}{t}$
where $t\in\mathbb{N}^{*}$, with high probability the smallest non
zero eigenvalue of the Laplacian is equal to $2-2\cos(\pi(2t+1)^{-1})+o(1)$.
This eigenvalue arises from a small subgraph isomorphic to a line of
size $t$ linked to the giant connected component by only one edge.
\end{abstract}

\section{Introduction}

\subsection{Introduction}

The Erd\H{o}s-R\'enyi  graph is the most natural and simple model for random
graph \cite{erd6s1960evolution}: from the complete graph with $N$ vertices,
every edge is kept with probability $p=p_{N}$ independently. The
asymptotic behavior has been extensively studied \cite{bollobas1998random,janson2011random} for various regimes of $p_{N}$ and has shown
a large number of phenomena, for example, the apparition of a component
of macroscopic size at $p_{N}>\frac{1}{N}$ or the connectivity of
the graph for $p_{N}>\frac{\log N}{N}$.

Together with the adjacency matrix, the Laplacian is a canonical way
to encode the structure of the graph. It appears in physics as the Hamiltonian of a quantum particle living on the graph or, in probability, as the generator associated
with the random walk on the graph.
Spectral analysis of the Laplacian of the Erdos-Reyni has recently
enjoyed new results. The bulk of the spectrum has been analyzed in \cite{huang2020spectral} down
to polynomial regimes and to $d\geq\sqrt{\log N}$ in \cite[Chapter 2]{rivier2023thesis}
where local laws down to optimal scale were shown for the Green function.
In particular, in the supercritical regime, $d\geq\log N$ the law
of the spectrum in the bulk can be characterized by the free-convolution
between a Gaussian law (coming from the diagonal matrix of the degree)
and a semi-circle law (coming from the adjacency matrix). The right
edge of the spectrum has been studied in \cite[Chapter 3]{rivier2023thesis}
where the author showed that the largest eigenvalues are matched with
the largest degrees in the graph (see also \cite{campbell2022extreme}).
In a series of papers
\cite{alt2021delocalization,alt2021extremal,alt2021poisson} the authors
had previously achieved similar results for the adjacency matrix For
the left edge of the spectrum of the Laplacian in the supercritical
regime $d\geq\log N$ a correspondence can be established between the smallest
eigenvalues and small-degree vertices (see \cite[Chapter 3]{rivier2023thesis})

In this paper, we study the spectrum of the Laplacian for subcritical
regimes $d\leq\log N$ and compute its smallest non-zero eigenvalue
$\lambda_{2}$ also called the spectral gap. Contrarily
to the supercritical regime, for $d\leq\log N$ the graph becomes
disconnected and small connected components begin to appear. We recall
that the Laplacian is a positive matrix with $\lambda_{1}=0$ whose
multiplicity is equal to the number of connected components. A natural
guess would be that $\lambda_{2}$
is obtained from the Laplacian restricted to the small disconnected
clusters. We prove that the smallest eigenvalues are created
by \emph{small clusters connected to the giant connected component
by only one edge}. Moreover among all these small clusters, the line
of maximal length is the optimal subgraph to minimize the eigenvalue.
We then obtain an explicit formula $\lambda_{2}=2-2\cos\left(\pi(2t_{*}+1)^{-1}\right)+o(1)$
with $t_{*} \in \mathbb{N}^*$ corresponding to the size of the largest but non-giant
disconnected component.

\subsection{Model}

For $G=(V,E)$ a graph with $V$ the set of vertices and $E$ the
set of edges we denote $L(G)\in\mathbb{R}^{|V|\times|V|}$ the Laplacian,
$A(G)\in\mathbb{R}^{|V|\times|V|}$ the adjacency matrix and $D(G)\in\mathbb{R}^{|V|\times|V|}$
the diagonal matrix
of the degree associated with the graph defined by
\[
L(G)=\sum_{e\in E}L(e),\qquad A(G)=\sum_{e\in E}A(e),\qquad L(G)=D(G)-A(G)
\]
where for every $e=(x,y)\in V^{2}$ 
\begin{align*}
L(e) & \coloneqq L((x,y))=(\mathrm{1}_{x}-\mathrm{1}_{y})(\mathrm{1}_{x}-\mathrm{1}_{y})^{*},\quad A(e)\coloneqq A((x,y))=\mathrm{1}_{x}1_{y}^{*}+\mathrm{1}_{y}\mathrm{1}_{x}^{*}.
\end{align*}

We consider the Erd\H{o}s-R\'enyi  graph $\mathbb{G}_{N,d_{N}}$ defined
with set of vertices $V=\{1,\cdots,N\}$ and set of edges $E$ where
each $(x,y)\in V^{2}$, $x\neq y$ is added to the graph at random,
independently and with probability $p_{N}=\frac{d_{N}}{N}$. We call
$d_{N}$ the mean degree of $\mathbb{G}_{N,d_{N}}$.

We say that $\Omega=(\Omega_{N})_{N\in\mathbb{N}}$, events on the
Erd\H{o}s-R\'enyi  graph of size $N$, occures with high probability if $\mathbb{P}(\Omega_{N})\rightarrow1$
as $N\rightarrow\infty$.

\subsection{Main results}

We are interested in the first non-zero eigenvalue of the Laplacian
that we denote 
\[
\lambda_{2}:=\min\bigl(\text{Spec}(L(\mathbb{G}_{N,d_{N}}))\setminus\{0\}\bigr).
\]
Here is our main result.
\begin{thm}
\label{thm:1} Let $t_{*}\in\mathbb{N}^{*}$, $\epsilon>0$ and $d_{N}$
such that 
\begin{align}
(1+\epsilon)\frac{\log N}{t_{*}+1}\leq d_{N}\leq(1-\epsilon)\frac{\log N}{t_{*}}.\label{equ:regimes}
\end{align}
As $N\rightarrow\infty$ with high probability, we have 
\begin{align}
\lambda_{2}=2-2\cos\left(\frac{\pi}{2t_{*}+1}\right)+O\left(\frac{1}{\log N}\right).\label{eq:master-equation}
\end{align}
\end{thm}

As we will see in the proof, the value of $\lambda_{2}$ is closely
linked with the existence of lines of length $t_{*}$ that are connected
to the rest of the graph by exactly one edge. The picture below summarizes
the various regimes covered.

\begin{figure}[ht!]
    \begin{center}
        \def\svgwidth{0.9\columnwidth} 
            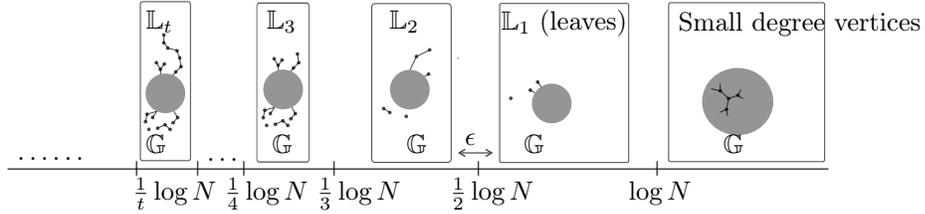
        \caption{Illustration of the regimes covered by Theorem \ref{thm:1}. 
        As the density parameter $d$ decreases, the graph becomes sparser and longer lines appear.
        Some of those lines are connected by exactly one edge to the main connected component and generate the formula \eqref{eq:master-equation} for $\lambda_2$. 
        The regimes corresponding to $d\geq \log N$ (small degree vertices) are not covered in this paper but a detailed analysis can be found in \cite[Chapter 3]{rivier2023thesis}.  }
        \label{fig:regimesThm:tree}
    \end{center}
\end{figure}

\section{Proof of Theorem \ref{thm:1}}

\subsection{Notations}

In the rest of the paper we simply write $\mathbb{G}=\mathbb{G}_{N,p_{N}}$
and $d=d_{N}$. We recall that a multiset is a set where elements can appear multiple times. We use the notation $\widehat{X}$ for
a multiset, $X$ for the associated set and for $x\in X$ we write
$\widehat{X}(x)$ the multiplicity of $x$. For example with $\widehat{X}=\{x,x,y\}$
we have $X=\{x,y\}$. $\widehat{X}(x)=2$, $\widehat{X}(y)=1$. We
also denote $\hat{X}\subset[[Y]]$ for $X\subset Y$.
\begin{defn}
For a graph $G=(V,E)$ and subgraph $T\subset V$ we define the anchor
$\widehat{\partial T}$ as the multiset 
\[
\widehat{\partial T}=\{x\in T\colon(x,y)\in E,y\notin T\}
\]
\end{defn}

It corresponds to the \textit{outgoing degree} of the vertices of $T$.

\begin{defn}
\label{def:weighted_Laplacian}For a graph $G=(V,E)$ and a multiset $\hat{X}\subset[[V]]$
we define the weighted Laplacian 
\[
L\bigl(G;\widehat{X}\bigr)=L(G)+\sum_{x\in\hat{X}}\mathrm{1}_{x}\mathrm{1}_{x}^{*}=L(G)+\sum_{x\in X}\hat{X}(x)\mathrm{1}_{x}\mathrm{1}_{x}^{*}.
\]
\end{defn}

\begin{rem}
Given $T\subset V$, the weigthed Laplacian corresponds to the usual
matrix restriction, $L(T;\widehat{\partial T})=L(G)|_{T}$.
\end{rem}

\subsection{Proof of Theorem \ref{thm:1}}

Through all the paper $t_{*}\in\mathbb{N}^{*}$, $\epsilon\in(0,\frac{1}{4})$
and as $-\tau\log\tau+\tau\rightarrow0$ for $\tau\rightarrow0$ we
can choose $\tau\in(0,1)$ small enough such that 
\begin{equation}
-\tau\log\tau+\tau\leq\frac{\epsilon}{2}.\label{eq:size-of-tau}
\end{equation}
Define 
\[
{\cal V}=\{x\in[N],D_{x}\leq\tau d\}.
\]
We call a line of size $t\in\mathbb{N}$ the graph 
\[
\mathbb{L}_{t}=\bigl(\{1,\cdots,t\},E\bigr)\text{ with }E=\{(i,i+1)\colon1\leq i<t\}
\]

\begin{prop}
\label{prop:XiProbSet} With high probability
\begin{enumerate}
\item $\mathbb{G}|_{{\cal V}}$ is a forest.
\item For any tree $T\subset\mathbb{G}|_{{\cal V}}$ we have $|T|\leq t_{*}$.\label{enu:largest_tree}
\item There exists a tree $T\subset\mathbb{G}|_{{\cal V}}$ that is isomorph
to a line of size $t_{*}$ and such that it is connected to $({\cal V})^{c}$
by exactly one edge attached at one of its two extremal points. \label{enu:LineExistence}
\end{enumerate}
\end{prop}

Denote $\mathfrak{F}$ the set of tree in the forest $\mathbb{G}|_{{\cal V}}$.
\begin{prop}
\label{prop:blockdiagonal} With high probability 
\begin{align*}
\sum_{\lambda\in\text{Spec}(L(\mathbb{G}))}\delta_{\lambda},\qquad\text{and}\qquad\delta_{0}+\sum_{T\in\mathfrak{F}}\sum_{\mu\in\text{Spec}(L(T;\widehat{\partial T}))}\delta_{\mu+\epsilon_{\mu}}
\end{align*}
agree on the interval $[0,\frac{1}{2}\tau d]$ and such that $\epsilon_{\mu}=O\bigl((1+\mu)d^{-1}\bigr)$
for all $\mu$.
\end{prop}

\begin{prop}
\label{prop:minimality} For any tree $T$ of size $|T|\leq t_{*}$,
$\widehat{X}\subset[[T]]$ and $\mu\in\text{Spec}(L(T;\widehat{X}))$
with $\mu>0$ we have
\begin{equation}
\mu\geq2-2\cos\left(\frac{\pi}{2t_{*}+1}\right).\label{eq:LineMin}
\end{equation}
 Moreover equality holds if and only if $T$ is a line of size $t_{*}$
and $\widehat{X}$ one of its extremal point and $\mu$ its smallest
eigenvalue.
\end{prop}

We prove Proposition \ref{prop:XiProbSet} in Section \ref{sec:Probability-estimate},
Proposition \ref{prop:blockdiagonal} in Section \ref{sec:Rigidity-results}
and Proposition \ref{prop:minimality} in Section \ref{sec:Spectrum-of-Laplacian}.
We then deduce Theorem \ref{thm:1} from these Propositions.
\begin{proof}[Proof of Theorem \ref{thm:1}]
 Because of Proposition \ref{prop:blockdiagonal} it is enough to
study the spectrum of $\text{Spec}(L(T;\widehat{\partial T}))$ with
$T\in\mathfrak{F}$ and $\widehat{\partial T}$ the associated anchor.
Using item \ref{enu:largest_tree} of Proposition \ref{prop:XiProbSet}
and Equation \eqref{eq:LineMin} of Proposition \ref{prop:minimality}
we obtain that with high probability 
\[
\lambda_{2}\geq2-2\cos\left(\frac{\pi}{2t_{*}+1}\right)+O\bigl(d^{-1}\bigr).
\]
On the other hand using item \ref{enu:LineExistence} of Proposition
\ref{prop:XiProbSet} and the equality case of Proposition \ref{prop:minimality}
we also have that with high probability 
\[
\lambda_{2}\leq2-2\cos\left(\frac{\pi}{2t_{*}+1}\right)+O\bigl(d^{-1}\bigr).
\]
\end{proof}

\section{Probability estimate on the graph, proof of Proposition \ref{prop:XiProbSet}}

\label{sec:Probability-estimate}This section is about properties
of the random graph that occur with high probability and that we
use in the following sections. Let us begin this section by setting
a few notations. We equip the graph $\mathbb{G}$ with its natural
graph distance $\text{dist}_{\mathbb{G}}.$ For $x\in[N]$ and $i\in\mathbb{N}$,
we define the sphere of radius $i$ around $x$ as 
\[
S_{i}(x):=\lbrace y\in[N]:\,\text{dist}_{\mathbb{G}}(x,y)=i\rbrace,
\]
as well as the balls of radius $i$ around $x$, $B_{i}(x):=\bigcup_{j\le i}S_{j}(x).$
\begin{lem}
With high probability, there exists a unique giant connected component
$\mathbb{G}_{cc}$ with size $N(1-o(1))$. Moreover $\mathbb{G}_{cc}^{c}\subset{\cal V}.$
\end{lem}

\begin{proof}
The giant connected component is a very classical result, see for
example \cite{bollobas1998random}. Also because \cite[Theorem 6.10]{bollobas1998random}
all the small connected component are of size $O(1)$, than every
very vertices of degree at least $\tau d$ belongs to $\mathbb{G}_{cc}$.
\end{proof}
We now state a few results from \cite{alt2021extremal}.
\begin{prop}
\label{prop:PreviousProba}For some ${\cal C}>0$ and $r<\frac{\epsilon\log N}{4(1+t_{*})^{2}\log d}$
with high probability, we have
\begin{enumerate}
\item $\forall x\in[N],D_{x}\leq{\cal C}d$, \label{enu:largest_degree}
\item $\Vert A(\mathbb{G})-\mathbb{E}(A(\mathbb{G})\Vert\le{\cal C}\sqrt{d}$
\label{enu:norm_adgency_matrix}
\item $|{\cal V}|\leq N^{1-\frac{1}{1+t_{*}}}$, \label{enu:V2_size}
\item For all $x\in{\cal V}$, $\mathbb{G}|_{B_{r}(x)}$ is a tree,\label{enu:proba-item-tree}
\item For all $x\in{\cal V}$, $|B_{r}(x)\cap{\cal V}|\leq t_{*}$.\label{enu:proba_no_large_cluster}
\end{enumerate}
\end{prop}

\begin{proof}[Proof of Proposition \ref{prop:PreviousProba}]
 Item \ref{enu:largest_degree} is stated in \cite[Lemma 3.3]{alt2021extremal}
and Item \ref{enu:norm_adgency_matrix} is stated in \cite[Corollary 2.3]{alt2021extremal}.
We now prove Item \ref{enu:V2_size}. We can apply \cite[Lemma D.2]{alt2021extremal}{]}
with $n=N-1$ and $\mu=d(1-N^{-1})$ and $a=\tau-1+O(N^{-1})$, we
find
\begin{equation}
\mathbb{P}(x\in{\cal V})\leq e^{-d(\tau\log\tau-\tau+1)}(1+O(N^{-1}))\leq N^{-\frac{1+\epsilon}{1+t_{*}}(1-\frac{\epsilon}{2})}(1+O(N^{-1}))=O(N^{-\nu}),\label{eq:bound-on-proba-V}
\end{equation}
where we chose $\tau$ as in \ref{eq:size-of-tau} and $\frac{(1+\epsilon)(1-\epsilon/2)}{1+t_{*}}\geq\nu\geq\frac{1+\epsilon/3}{1+t_{*}},$for
$\epsilon\leq1/3.$ We find that for a $C\geq0$ large enough
\[
\mathbb{E}|{\cal V}|=N\mathbb{P}(x\in{\cal V})\leq CN^{1-\nu}
\]
so that by Markov $\mathbb{P}(|{\cal V}|\geq N^{1-\frac{1}{1+t_{*}}})\leq CN^{\frac{1}{1+t_{*}}-\nu}=O\bigl(N^{-\frac{\epsilon}{3(1+t_{*})}}\bigr)$.

To prove Item \ref{enu:proba-item-tree} we use \cite[Lemma 5.5]{alt2021extremal}
with $k=1$ and have that for any $x\in[N]$, 
\begin{equation*}
    \mathbb{P}(x\in{\cal V},\,B_{r}(x)\text{ contain a cycle })\leq2r^{2}(Cd)^{2r+1}N^{-1-\frac{1}{t_{*}+1}}=o\bigl(N^{-1}\bigr),
\end{equation*}
where we use that $(2r+1)\log d+2\log r\leq4r\log d\leq\frac{4\epsilon\log N}{(1+t_{*})^{2}}\leq\frac{4}{5(t_{*}+1)}\log N$.
Using a union bound, we conclude that 
\begin{equation*}
    \mathbb{P}(\exists x\in[N],\,x\in{\cal V},\,B_{r}(x)\text{ contain a cycle })=o(1).
\end{equation*}

To prove Item \ref{enu:proba_no_large_cluster} we explore the neighbourhood
of $x$ adding the set $S_{i}(x)$ for $i\leq r$ one after the other.
We denote the set of edges 
\[
E_{i+1}=E(\mathbb{G}(B_{i+1}))\setminus E(\mathbb{G}(B_{i}))\cup E(\mathbb{G}(S_{i}))
\]
and the generated sigma algebra ${\cal F}_{i}=\sigma(E_{1},\cdots,E_{i})$.
This corresponds to $\mathbb{G}(B_{i})$ without the set of edges
$\{(x,y)\in E:x,y\in S_{i}\}$. Remark that conditionnally on ${\cal F}_{i}$,
$E_{i+1}$ is given by a set of $|S_{i}|\times(N-|B_{i}|)$ independent
random Bernoulli variables of parameter $\frac{d}{N}$.

For $y\in S_{i}$, we define the random variables $\tilde{D}_{y}:=|S_{1}(y)\cap S_{i+1}|$
which form a family of independent random variable with $\tilde{D}_{y}\leq D_{y}-1$.
We introduce $V_{i}=|S_{i}\cap{\cal V}|$ and, using \eqref{eq:bound-on-proba-V},
we find that, for any $m_{1},\ldots,m_{r}\in\mathbb{N},$
\begin{align*}
1_{|S_{i}|\leq({\cal C}d)^{i}}\mathbb{P}(V_{i} 
& \geq m_{i}|{\cal F}_{i})\leq1_{|S_{i}|\leq({\cal C}d)^{i}}
\sum_{\{y_{1},\cdots,y_{m_{i}}\}\subset S_{i}}\mathbb{P}\Bigl(\cap_{j\leq m_{i}}\{\tilde{D}_{y_{j}}\leq\tau d-1\}|{\cal F}_{i}\Bigr)\\
 & \leq({\cal C}d)^{im_{i}}N^{-m_{i}\nu}.
\end{align*}
We denote ${\cal D}(m)=\{(m_{i})\in\mathbb{N}^{r}:\sum_{i=1}^{r}m_{i}=m\}$
and we have 
\begin{align*}
1_{\forall y\in B_{r}(x),D_{y}\leq{\cal C}d}\mathbb{P}(|B_{r}(x)\cap{\cal V}|\geq m|{\cal F}_{r}) & \leq1_{\forall y\in B_{r}(x),D_{y}\leq{\cal C}d}\sum_{{\cal D}(m)}\mathbb{P}\Bigl(\cap_{i=1}^{r}\{V_{i}\geq m_{i}\}|{\cal F}_{r}\Bigr)\\
 & \leq\sum_{{\cal D}(m)}\prod_{i=1}^{r} 1_{|S_{i}|\leq({\cal C}d)^{i}}\mathbb{P}(V_{i}\geq m_{i}|{\cal F}_{i})\\
 & \leq r^{m}({\cal C}d){}^{rm}N^{-m\nu}
\end{align*}
Therefore for $m=t_{*}+1$ we have $N^{-m\nu}\leq N^{-(1+\frac{\epsilon}{3})}$
and then 
\begin{equation*}
    \mathbb{P}(\forall y\in[N],D_{y}\leq{\cal C}d\cap|B_{r}(x)\cap{\cal V}|\geq t_{*}+1)\leq r^{t_{*}+1}(Cd){}^{r(t_{*}+1)}N^{-(1+\frac{\epsilon}{3})}=o(N^{-1})
\end{equation*}
where we use that $(1+t_{*})r\log d<\frac{\epsilon\log N}{4(1+t_{*})}$.
Using a union bound, we conclude that 
\[
\mathbb{P}(\exists x\in[N],\ |B_{r}(x)\cap{\cal V}|\geq t_{*}+1)=o(1).
\]
\end{proof}
With Proposition \ref{prop:PreviousProba} we obtain the two first
items of Proposition \ref{prop:XiProbSet}. We now prove the last
item.
\begin{proof}[Proof of Proposition \ref{prop:XiProbSet}, item \ref{enu:LineExistence}]
 We denote ${\cal N}$ the number of tree $T\subset\mathbb{G}|_{{\cal V}}$
that is isomorph to a line of size $t_{*}$ and such that it is connected
to ${\cal V}{}^{c}$ by exactly one edge attached at one of its extremal
point. We use a second-moment method and will prove that 
\begin{align}
\mathbb{E}({\cal N}) & =Nd^{t_{*}}e^{-t_{*}d}\left(1+O\left(\frac{t_{*}^{2}d}{N}\right)\right)\nonumber \\
\mathbb{E}({\cal N}^{2})-\mathbb{E}({\cal N})^{2} & =O\left(\mathbb{E}({\cal N})+1+\frac{t_{*}^{2}d\mathbb{E}({\cal N})^{2}}{N}\right)\label{eq:variance_Line}
\end{align}
Assuming Equation (\ref{eq:variance_Line}) we have $\mathbb{E}({\cal N})\geq N^{1-(1-\epsilon)\frac{t_{*}}{t_{*}}+o(1)}\rightarrow\infty$
where we use \eqref{equ:regimes} and then by Chebyshev's inequality
\[
\mathbb{P}({\cal N}=0)\leq\mathbb{P}\Bigl({\cal N}\leq\frac{\mathbb{E}({\cal N})}{2}\Bigr)
\leq\frac{4\mathbb{E}({\cal N}-\mathbb{E}{\cal N})^{2}   }{\mathbb{E}({\cal N})^{2}}=O\Bigl(\frac{1}{\mathbb{E}({\cal N})}\Bigr)=o(1).
\]
We now prove Equation (\ref{eq:variance_Line}). For $Y=\{y_{1},\cdots,y_{t_{*}},x\}\subset[N]$
, ${\cal L}_{Y}$ the event that $\mathbb{G}\vert_{\{y_{1},\cdots,y_{t_{*}}\}}$
is a line and that $(y_{t_{*}},x)$ is the only edge between $\{y_{1},\cdots,y_{t_{*}}\}$
and $\{y_{1},\cdots,y_{t_{*}}\}^{c}$. Then 
\[
\mathbb{P}({\cal L}_{Y})=\frac{d^{t_{*}}}{N^{t_{*}}}\Bigl(1-\frac{d}{N}\Bigr)^{t_{*}(N-t_{*})+1}
=\frac{d^{t_{*}}}{N^{t_{*}}}e^{-dt^{*}}\Bigl(1+O\Bigl(\frac{t_{*}^{2}d}{N}\Bigr)\Bigr)
\]
Then 
\[
\mathbb{E}({\cal N})=\sum_{Y}\mathbb{P}({\cal L}_{Y})=Nd^{t_{*}}e^{-dt^{*}}\Bigl(1+O\Bigl(\frac{t^{2}d}{N}\Bigr)\Bigr).
\]
We also have 
\[
\mathbb{E}({\cal N}^{2})-\mathbb{E}({\cal N})^{2}=\sum_{Y_{1},Y_{2}}\mathbb{P}({\cal L}_{Y_{1}}\cap{\cal L}_{Y_{2}})-\mathbb{P}({\cal L}_{Y_{1}})\mathbb{P}({\cal L}_{Y_{2}})
\]
If $Y_{1}$ and $Y_{2}$ are disjoint, the probability that there
is no edges between $Y_{1}$ and $Y_{2}$ should be added only once
and we have 
\[
\mathbb{P}({\cal L}_{Y_{1}}\cap{\cal L}_{Y_{2}})=\mathbb{P}({\cal L}_{Y_{1}})\mathbb{P}({\cal L}_{Y_{2}})\Bigl(1-\frac{d}{N}\Bigr)^{-t_{*}^{2}}
\]
Therefore
\[
\sum_{Y_{1},Y_{2},\text{ disjoint }}\mathbb{P}({\cal L}_{Y_{1}}\cap{\cal L}_{Y_{2}})-\mathbb{P}({\cal L}_{Y_{1}})\mathbb{P}({\cal L}_{Y_{2}})=O\Bigl(\frac{dt^{2}}{N}\mathbb{E}({\cal N})^{2}\Bigr)
\]
If $Y_{1}\cap Y_{2}\neq\emptyset$ and $Y_{1}\neq Y_{2}$ we have
that then $Y_{1}\cup Y_{2}$ form a connected component in ${\cal V}$
and then 
\begin{align*}
\sum_{Y_{1}\neq Y_{2},Y_{1}\cap Y_{2}\neq\emptyset}\mathbb{P}({\cal L}_{Y_{1}}\cap{\cal L}_{Y_{2}}) & \leq t_{*}^{2}\mathbb{E}\bigl(\big\vert\{U\text{ connected component in }{\cal V},\,|U|\geq t_{*}+1\}\big\vert\bigr)\\
 & =o(1),
\end{align*}
as in the proof of Item \ref{enu:proba_no_large_cluster} of Proposition
\ref{prop:PreviousProba}. Finally for the case $Y_{1}\cap Y_{2}\neq\emptyset$
and $Y_{1}=Y_{2}$ we just have
\[
\sum_{Y_{1},Y_{2},Y_{1}=Y_{2}}\mathbb{P}({\cal L}_{Y_{1}}\cap{\cal L}_{Y_{2}})=\mathbb{E}({\cal N}).
\]
and Equation (\ref{eq:variance_Line}) follows from the three above
estimates.
\end{proof}

\section{Rigidity results, proof of Proposition \ref{prop:blockdiagonal}}

\label{sec:Rigidity-results}We analyze the spectrum of the graph
locally around the vertices in ${\cal V}.$ We will work on small
clusters of such vertices. Let us start by setting a few notations.
For $M\coloneqq(M_{xy})_{x,y\in[n]}\in\mathbb{R}^{n\times n}$ and
$Y\subseteq[n]$ we denote by $M\vert_{Y}=(M_{xy})_{x,y\in Y}$ the
restriction of $M$ to $Y.$ We write $\lambda_{1}(M)\leq\lambda_{2}(M)\leq\ldots\leq\lambda_{n}(M)$
the increasing ordering of the eigenvalues of $M.$ We use the standard
$\ell^{2}$ norm, $\Vert\cdot\Vert\coloneqq\Vert\cdot\Vert_{2}$.
\begin{defn}
\label{def:U-connected-components}We denote $\mathfrak{U}$ the connected
components of $\cup_{x\in{\cal V}}B_{2}(x)$.
\end{defn}

We define some good properties of the graph.
\begin{defn}
\label{def:Xi} Let $\Xi_{1}$ a probability set such that
\begin{enumerate}
\item For $x\in[N]$, $D_{x}\leq{\cal C}d$.
\item For all $U\in\mathfrak{U}$, $|{\cal V}\cap U|\leq t_{*}$.
\item \label{enu:The-connected-components}The graph restricted to $\cup_{x\in{\cal V}}B_{3}(x)$
is a forest.
\end{enumerate}
Let $\Xi_{2}$ a probability set such that
\begin{enumerate}
\item For $x\in[N]$, $D_{x}\leq{\cal C}d$.
\item $\Vert A(\mathbb{G})-\mathbb{E}(A(\mathbb{G})\Vert\le{\cal C}\sqrt{d}$
\item $|{\cal V}|\leq N^{1-\frac{1}{1+t_{*}}}$,
\end{enumerate}
\end{defn}

These properties occur with high probability and we will assume
that they are satisfied throughout
this section.
\begin{prop}
\label{prop:Xi}$\mathbb{P}(\Xi_{1}\cap\Xi_{2})\geq1-o(1)$
\end{prop}

\begin{proof}[Proof of Proposition \ref{prop:Xi}]
If $U$ is a cluster with $t_{*}+1$ vertices in ${\cal V}$, then
there is a chain $y_{1},\cdots,y_{t_{*}+1}\in\mathcal{V}$ such that
for any $1<i\leq t_{*}+1$, there exists $j<i$ with $\text{dist}_{\mathbb{G}}(y_{i},y_{j})\leq5$.
Therefore $y_{j}\in B_{y_{1}}(5t_{*})$ for all $j\leq t_{*}+1$ and
we conclude that the probability is small by Proposition \ref{prop:PreviousProba},
item \ref{enu:proba_no_large_cluster}. The other items in $\Xi_{1}$
and $\Xi_{2}$ also follow from Proposition \ref{prop:PreviousProba}.
\end{proof}
\begin{prop}
\label{prop:Tree_to_Ball}On the event $\Xi_{1}$, for any $U\in\mathfrak{U}$
and $F={\cal V}\cap U$, 
\[
\sum_{\lambda\in\text{Spec}(L(\mathbb{G})|_{U})}\delta_{\lambda},\qquad\text{and}\qquad\sum_{\mu\in\text{Spec}(L(F;\widehat{\partial F}))}\delta_{\mu+\epsilon_{\mu}}
\]

agree on the interval $[0,\frac{\tau}{2}d]$ and such that $\epsilon_{\mu}=O\bigl((\mu+1)d^{-1}\bigr)$
for all $\mu$.
\end{prop}

\begin{proof}[Proof of Proposition \ref{prop:Tree_to_Ball}]
 Let us consider a connected component $U\in\mathcal{U}$ and the
restriction of the Laplacian to this connected component $H:=L(\mathbb{G})\vert_{U}.$
Let $A:=L(\mathbb{G})\vert_{F}$, $C:=L(\mathbb{G})\vert_{U\setminus F}$
and $B\in\lbrace0,-1\rbrace^{\vert F\vert\times\vert U\setminus F\vert}$
such that 
\[
H=\begin{pmatrix}A & B^{*}\\
B & C
\end{pmatrix}.
\]

Then recalling Definition \ref{def:weighted_Laplacian}, we see that
$A=L(F,\widehat{\partial F})$. Note that $A$ may be itself a block
matrix if $F$ is made up of disjoint trees. Moreover the matrix $C$
is given by the sum of a diagonal matrix with entries in the interval
$[\tau d,+\infty)$ and the adjacency matrix of a forest with maximal
degree ${\cal C}d$ by item \ref{enu:The-connected-components} of
$\Xi_{1}$. By perturbation theory and Lemma \ref{lem:adj_of_tree}
we know that 
\begin{equation}
\ensuremath{\mathrm{\mathrm{spec}}C}\subseteq\Bigl[\tau d-2\sqrt{{\cal C}d},\infty\Bigr)\subseteq\bigl[3\tau d/4,\infty\bigr)\label{eq:Bound_C_spectrum}
\end{equation}
for $d$ large enough. In addition, since $B$ is given as the
projection of the adjacency matrix of a collection of disjoint stars
(rooted at vertices of $\partial F$), we know that $\Vert B\Vert\leq\sqrt{\tau d}$.

We define 
\[
H(s)=\begin{pmatrix}A & sB^{*}\\
sB & C
\end{pmatrix}\qquad\text{for }s\in[0,1].
\]
Here $H(1)=H$ and $H(0)$ is block diagonal with the matrices $A$
and $C$ therefore $\ensuremath{\mathrm{\mathrm{spec}}(H(0))}\cap[0,\frac{3}{4}\tau d)=\ensuremath{\mathrm{\mathrm{spec}}(A)}\cap[0,\frac{3}{4}\tau d)$.
We write $\lambda_{1}(s)\leq\lambda_{2}(s)\leq\cdots$ the eigenvalues
of $H(s)$. By perturbation theory, these are Lipschitz functions.
More precisely suppose $\lambda_{i}(s)\in\mathrm{spec}(H(s))\cap[0,\frac{\tau}{2}d]$,
write $u(s)\in\mathbb{R^{\vert U\vert}}$ the corresponding normalized
eigenvector and denote by $u_{1}(s):=u(s)\vert_{F}$ and $u_{2}(s):=u(s)\vert_{U\setminus F}$.
We have 
\begin{align}
\frac{d}{ds}\lambda_{i}(s) & =\Big\langle u(s),\left[\frac{d}{ds}H(s)\right]u(s)\Big\rangle=2\big\langle u_{1}(s),B^{*}u_{2}(s)\big\rangle\label{eq:Lipshitz_eigenvalue}
\end{align}
where we use that $2\langle\frac{d}{ds}u(s),H(s)u(s)\rangle=\lambda(s)\frac{d}{ds}\|u(s)\|^{2}=0$.
In the case of a degenerate eigenvalue, $u(s)$ is not well defined. However
one can still choose a normalized vector in the eigenspace such that
the above equation is true (following the classical proof that the
eigenvalues are Lipschitz using min-max principle).

We will prove that 
\begin{equation}
\Big\vert\frac{d}{ds}\lambda_{i}(s)\Big\vert\leq\frac{8t_{*}}{\tau d}(\lambda_{i}(s)+4),\qquad s\in[0,1],\label{eq:unif-bound-lambdaprime}
\end{equation}
from which we will conclude that 
\[
\vert\lambda-\mu\vert=|\lambda_{i}(0)-\lambda_{i}(1)|=O\bigl((\lambda_{i}(0)+1)d^{-1}\bigr).
\]
Recall here that $\lambda_{i}(0)\in\mathrm{\mathrm{spec}}(A)=\mathrm{\mathrm{spec}}(L(F;\widehat{\partial F}))$
and $\lambda_{i}(1)\in\mathrm{\mathrm{spec}}(L(\mathbb{G})|_{U})$.

Let us fix $s\in[0,1]$ and write $\lambda=\lambda(s)$ and $u_{i}=u_{i}(s)$
in the following computations. The eigenvalue-eigenvector equation
can be written in blocks as 
\[
\begin{pmatrix}A-\lambda & sB^{*}\\
sB & C-\lambda
\end{pmatrix}\begin{pmatrix}u_{1}\\
u_{2}
\end{pmatrix}=0.
\]
Using (\ref{eq:Bound_C_spectrum}), $\Vert B\Vert\leq\sqrt{\tau d}$
and standard perturbation theory, we see that $$\min_{\nu\in\mathrm{\mathrm{spec}}(C)}|\lambda-\nu|\geq\frac{\tau d}{4}.$$
The matrix $C-\lambda$ is thus invertible with $\|(C-\lambda)^{-1}\|\leq\frac{\tau d}{4}$.
After a couple of substitutions, the eigenvalue-eigenvector equation
becomes 
\begin{equation}
\begin{cases}
u_{2}=-s(C-\lambda)^{-1}Bu_{1} & ,\\
(A-s^{2}B^{*}(C-\lambda)^{-1}B)u_{1}=\lambda u_{1} & .
\end{cases}\label{eq:block_equa_lambda_1_2}
\end{equation}
\label{eq:block-eequ-lambda}Replacing $A$ by the deformed Laplacian
matrix and applying $u_{1}^{*}$ to the second line yields 
\[
\lambda\|u_{1}\|^{2}=\langle u_{1},L(F)u_{1}\rangle+\sum_{x\in\partial F}\widehat{\partial F}(x)u_{1}(x)^{2}-s^{2}\big\langle u_{1},(B^{*}(C-\lambda)^{-1}B)u_{1}\big\rangle.
\]
and then 
\begin{align}
\sum_{x\in\partial F}\widehat{\partial F}(x)u_{1}(x)^{2}\leq\lambda\|u_{1}\|^{2}+s^{2}\big\vert\big\langle u_{1},(B^{*}(C-\lambda)^{-1}B)u_{1}\big\rangle \big\vert \leq(\lambda+4)\|u_{1}\|^{2}\label{eq:eval-rigidity}
\end{align}
where we use the fact that $L(F)$ is positive and that $\|B^{*}(C-\lambda)^{-1}B\|\leq\Vert B\Vert^{2}\Vert(C-\lambda)^{-1}\Vert\leq4$.
Finally, applying $\langle u_{1}^{*},B^{*}\cdot\rangle$ to the first
line of \eqref{eq:block_equa_lambda_1_2}, we get 
\begin{align*}
\big\vert\big\langle u_{1},B^{*}u_{2}\big\rangle\big\vert & =\big\vert\big\langle u_{1},B^{*}(C-\lambda)^{-1}Bu_{1}\big\rangle\big\vert\\
 & =\Big\vert \sum_{x,y\in\partial F}u_{1}(x)u_{1}(y)\big\langle B1_{y},(C-\lambda)^{-1}B1_{x}\big\rangle\Big\vert\\
 & \leq\sum_{x,y\in\partial F}|u_{1}(x)||u_{1}(y)|\frac{4\sqrt{\widehat{\partial F}(x)\widehat{\partial F}(y)}}{\tau d}\\
 & =\frac{4}{\tau d}\left(\sum_{x\in\partial F}|u_{1}(x)|\sqrt{\widehat{\partial F}(x)}\right)^{2}\\
 & \leq\frac{4t_{*}}{\tau d}(\lambda+4)\|u_{1}\|^{2}.
\end{align*}
by Cauchy-Schwartz and \eqref{eq:eval-rigidity}. This proves \eqref{eq:unif-bound-lambdaprime}
and concludes the proof.
\end{proof}
\begin{prop}
\label{prop:Exponential_decay}On the event $\Xi_{1}$, for any $U\in\mathfrak{U}$,
$\lambda\in[0,\frac{\tau}{2}d]\cap\text{Spec}(L(\mathbb{G}|_{U}))$
and $\boldsymbol{u}$ the associated eigenvector we have 
\[
\|\boldsymbol{u}|_{\partial U}\|=O\bigl(d^{-1}\bigr)
\]
\end{prop}

\begin{proof}[Proof of Proposition \ref{prop:Exponential_decay}]
We use the same notation as for the proof of Proposition \ref{prop:Tree_to_Ball}.
Setting $s=1$ in \ref{eq:block-eequ-lambda}, we have 
\[
u_{2}=-(C-\lambda)^{-1}Bu_{1}.
\]
Let us write $C=D-A(\mathbb{G}|_{U\setminus F})$ with with $D:=D(\mathbb{G})|_{U\setminus F}$.
The second resolvent identity reads
\begin{align*}
(C-\lambda)^{-1} & =(D-\lambda)^{-1}+(C-\lambda)^{-1}A(\mathbb{G}|_{U\setminus F})(D-\lambda)^{-1},
\end{align*}
and because $\text{supp}((D-\lambda)^{-1}Bu_{1})\subset\cup_{x\in F}S_{1}(x)$,
we have $$\text{supp}\bigl((D-\lambda)^{-1}Bu_{1}\cap\partial U\bigr)=\emptyset.$$
We thus find that 
\[
u_{2}|_{\partial U}=-\left((C-\lambda)^{-1}A(\mathbb{G}|_{U\setminus F})(D-\lambda)^{-1}Bu_{1}\right)\big\vert_{\partial U}
\]
The graph $\mathbb{G}|_{U\setminus F}$ is a tree so by Lemma \ref{lem:adj_of_tree},
we have $\|A(\mathbb{G}|_{U\setminus F})\|\leq2\sqrt{{\cal C}d}$.We
can use the same bounds for $\Vert B\Vert,$ $\|(D-\lambda)^{-1}\|$
and $\|(C-\lambda)^{-1}\|$ as in the proof of Proposition \ref{prop:Tree_to_Ball}.
Putting everything together, we find 
\[
\|\boldsymbol{u}|_{\partial U}\|=\|u_{2}|_{\partial U}\|\leq\frac{16\sqrt{{\cal C}}}{\tau^{3/2}d}\|u_{1}\|.
\]
\end{proof}
\begin{prop}
\label{prop:trivialevect-1}On the event $\Xi_{2},$ the vector $\mathbf{q}:=\frac{1}{\sqrt{|Y|}}\mathrm{1_{Y}},$where
$Y:=[N]\setminus\bigcup_{x\in\mathcal{V}}B_{3}(x)$ satisfies for
any $c<\frac{1}{1+t_{*}}$
\[
\|\mathbf{e}-\mathbf{q}\Vert^{2}=O(N^{-c}),\qquad\langle\mathbf{q,}L(\mathbb{G})\mathbf{q}\rangle=O(N^{-c}).
\]
\end{prop}

\begin{proof}[Proof of Proposition \ref{prop:trivialevect-1}]
On the event $\Xi_{2}$,
we have 
\[
\vert Y^{c}\vert\leq\Big\vert\bigcup_{x\in\mathcal{V}}B_{3}(x)\Big\vert\leq(\mathcal{C}d)^{3}\vert{\cal V\vert}\leq\mathcal{C}^{5}d^{4}N^{1-\frac{1}{1+t_{*}}}=O(N^{1-c}),
\]
for $c<\frac{1}{1+t_{*}}.$ In particular, $\vert Y\vert=(1-o(1))N.$
It is a straightforward computation to see that on $\Xi_{1}$
\[
\langle\mathbf{q},L\mathbf{q}\rangle=\frac{1}{\vert Y\vert}\sum_{y\in\bigcup_{x\in\mathcal{V}}S_{4}(x)}1\leq\frac{\mathcal{C}d\vert Y^{c}\vert}{\vert Y\vert}=O(N^{-c}),
\]
and
\[
\mathbf{e-\mathbf{q}}=\frac{1}{\sqrt{N}}\sum_{y\notin Y}1_{y}+\sum_{z\in Y}\Bigl(\frac{1}{\sqrt{N}}-\frac{1}{\sqrt{|Y|}}\Bigr)1_{z},
\]
and thus 
\[
\Vert\mathbf{e}-\mathbf{q}\Vert^{2}=\frac{|Y^{c}|}{N}+\frac{\bigl(\sqrt{|Y|}-\sqrt{N}\bigr)^{2}}{N}=O(N^{-c}).
\]
\end{proof}
\begin{prop}
\label{prop:Spec_Vc}On the event $\Xi_{2},$ $\mathrm{\mathrm{spec}(L(\mathbb{G})\vert_{\mathcal{V}^{c}})\subset\lbrace\mu\rbrace\cup[\tau d/2,}+\infty)$
for some $\mu=O(N^{-c})$, $0<c<\frac{1}{1+t_{*}}$.
\end{prop}

\begin{proof}[Proof of Proposition \ref{prop:Spec_Vc}]
We have 
\[
L(\mathbb{G})=D(\mathbb{G})-(A(\mathbb{G})-\mathbb{E}(A(\mathbb{G}))-\mathbb{E}(A(\mathbb{G})).
\]
By definition of $D$ and ${\cal V}$ we have $\lambda_{1}(D(\mathbb{G})|_{{\cal V}})\geq\tau d$
and with Definition \ref{def:Xi} we have that $\Vert A(\mathbb{G})-\mathbb{E}(A(\mathbb{G})\Vert=O(\sqrt{d})$.
Therefore,
\[
\lambda_{1}\bigl((D(\mathbb{G})-(A(\mathbb{G})-\mathbb{E}(A(\mathbb{G})))|_{{\cal V}^{c}}\bigr)\geq\tau d-O(\sqrt{d}).
\]
Because $\mathbb{E}(A(\mathbb{G}))|_{{\cal V}^{c}}=\frac{d}{N}1_{{\cal V}^{c}}1_{{\cal V}^{c}}^{*}$
is a rank one matrix, by interlacing $\lambda_{2}(L(\mathbb{G})|_{{\cal V}^{c}})\geq\tau d-O(\sqrt{d})$.
Moreover using Proposition \ref{prop:trivialevect-1} and Lemma \ref{lem:PerturbationLemma},
we have $\lambda_{1}(L(\mathbb{G})|_{{\cal V}^{c}})=O(N^{-c})$ for
any constant $c>0$ small enough.
\end{proof}
We now have all the ingredients to prove Proposition \ref{prop:blockdiagonal}.

\begin{proof}[Proof of Proposition \ref{prop:blockdiagonal}]
 We work on $\Xi_{1}\cap\Xi_{2}$ that is a high probability event
by Proposition \ref{prop:Xi}.

We denote ${\cal U}=\cup_{x\in{\cal V}}B_{2}(x)=\bigcup_{U\in\mathfrak{U}}U$
and recall the definition of $\mathfrak{U}$ and $\mathbf{q}$ from
Definition \ref{def:U-connected-components} and Proposition \ref{prop:trivialevect-1}
respectively. We define three vector spaces $E_{1}=\text{Span}(\boldsymbol{q})$,
$E_{2}=\text{Span}(\{\boldsymbol{1}_{y},y\in{\cal U}\})$ and $E_{3}=(E_{1}+E_{2})^{\perp}$.
Let $(v_{i},\lambda_{i})_{i\le|{\cal U}|}$ the eigenvectors-eigenvalues
of $L(\mathbb{G})|_{{\cal U}}$ that form an orthonormal basis of
$E_{2}$ that decompose into $E_{2}=E_{2}^{\leq}+E_{2}^{>}$ where
$E_{2}^{\leq}=\text{Span}(\{v_{i}\colon\lambda_{i}\leq\frac{1}{2}\tau d\})$
and $E_{2}^{>}=\text{Span}(\{v_{i}\colon\lambda_{i}>\frac{1}{2}\tau d\})$.
We complete $\{\boldsymbol{q},v_{1},\cdots,v_{|{\cal U}|}\}$ into
an orthogonal basis for $\mathbb{R}^{[N]}=E_{1}+E_{2}^{\leq}+E_{2}^{>}+E_{3}$
to obtain an orthogonal matrix $V$. The Laplacian of the graph in
this basis is in the form of a block matrix 
\[
V^{*}L(\mathbb{G})V=\begin{pmatrix}\nu & 0 & 0 & X_{\nu}^{*}\\
0 & D^{\le  q} & 0 & X^{\leq}\\
0 & 0 & D^{>} & X^{>}\\
X_{\nu} & X^{\leq} & X^{>} & Y
\end{pmatrix}
\]

We explain every element of this matrix :
\begin{itemize}
\item $\nu=\langle\mathbf{q,}L(\mathbb{G})\mathbf{q}\rangle=O(N^{-c})$
by Proposition \ref{prop:trivialevect-1}, for $c<\frac{1}{1+t_{*}}$.
\item Because $\text{supp}(\boldsymbol{q})\subseteq(\bigcup_{x\in\mathcal{V}_{2}}B_{3}(x))^{c}$
we have $\text{supp}(L(\mathbb{G})\mathbf{q})\subseteq(\bigcup_{x\in\mathcal{V}_{2}}B_{2}(x))^{c}$
and then $L(\mathbb{G})\mathbf{q}\in E_{2}^{\perp}$, so we have the
two zeros on the first line and column.
\item $X_{\nu}$ is a 1-column matrix that can be identified as a vector
such that $L(\mathbb{G})\mathbf{q}=\nu\mathbf{q}+X_{\nu}$. We have
\[
\|X_{\nu}\|\leq\|L(\mathbb{G})\mathbf{q}\|=\|L(\mathbb{G})(\mathbf{q}-\boldsymbol{e})\|\leq\|L(\mathbb{G})\|\|(\mathbf{q}-\boldsymbol{e})\|=O(N^{-c+o(1)})
\]
where we use that $L(\mathbb{G})\boldsymbol{e}=0$ and Propsition
\ref{prop:trivialevect-1}.
\item $D^{\leq}$ and $D^{>}$ are diagonal matrices with entries $\{\lambda_{i}\colon\lambda_{i}\leq\frac{1}{2}\tau d\}$
and $\{\lambda_{i}\colon\lambda_{i}>\frac{1}{2}\tau d\}$. This directly
follows from the choice of $v_{i}$ as the eigenvectors $L(\mathbb{G})|_{{\cal U}}$.
\item $X^{\leq}$ and $X^{>}$ describe the adjancy matrix$-A(\mathbb{G})$
between the set ${\cal U}$ and ${\cal U}^{c}$ that is a tree by
definition of $\Xi_{1}$ and we have $\max\{\|X^{\leq}\|,\|X^{>}\|\}\leq2\sqrt{{\cal C}d}$
by Lemma \ref{lem:adj_of_tree}.
\item $Y=L(\mathbb{G})|_{E_{3}}$. We have $E_{1}+E_{3}=E_{2}^{\perp}=\text{Span}(1_{x},x\in{\cal U}^{c})$.
Then if we remove the block matrices associated to the space $E_{2}$
we obtain 
\[
V^{*}L(\mathbb{G})|_{{\cal U}^{c}}V=\begin{pmatrix}\nu & X_{\nu}^{*}\\
X_{\nu} & Y
\end{pmatrix}=\begin{pmatrix}\nu & 0\\
0 & Y
\end{pmatrix}+\begin{pmatrix}0 & X_{\nu}^{*}\\
X_{\nu} & 0
\end{pmatrix},
\]
where we write $V$ instead of $V|_{{\cal U}^{c}}$and simplify the
notation by omitting the zero blocks. We immediately conclude that
\[
\lambda_{2}(L(\mathbb{G})|_{{\cal U}^{c}})\geq\max\{\nu,\lambda_{1}(Y)\}-\|X_{\nu}\|.
\]
We have Proposition \ref{prop:Spec_Vc} and then use that $\text{Span}(1_{x},x\in{\cal U}^{c})\subset\text{Span}(1_{x},x\in{\cal V}^{c})$
so by interlacing we obtain 
\[
\frac{1}{2}\tau d\leq\lambda_{2}(L(\mathbb{G})|_{{\cal V}^{c}})\leq\lambda_{2}(L(\mathbb{G})|_{{\cal U}^{c}}).
\]
Because of the previous estimate on $\nu$ and $\|X_{\nu}\|$ we finally
conclude that 
\[
\lambda_{1}(Y)\geq\frac{1}{2}\tau d-o(1).
\]
\item Finally we improve the bound for $\|X^{\leq}\|$. Let $u_{2}\in E_{2}^{\leq},u_{3}\in E_{3}$,
we have 
\[
|\langle L(\mathbb{G})u_{3},u_{2}\rangle|=|\langle A(\mathbb{G})u_{3},u_{2}|_{\partial{\cal U}}\rangle|\leq2\sqrt{{\cal C}d}\|u_{3}\|\|u_{2}|_{\partial{\cal U}}\|
\]
and write $u_{2}=\sum_{i}\alpha_{i}v_{i}$ for some $\alpha_{i}\in\mathbb{R}$.
We have 
\begin{align*}
\|u_{2}|_{\partial{\cal U}}\|^{2} & =\sum_{U\in\mathfrak{U}}\|u_{2}|_{\partial U}\|^{2}\\
 & =\sum_{U\in\mathfrak{U}}\|\sum_{\text{Supp}(v_{i})\subset U}\alpha_{i}v_{i}|_{\partial U}\|^{2}\\
 & \leq\sum_{U\in\mathfrak{U}}t_{*}\sum_{\text{Supp}(v_{i})\subset U}\alpha_{i}^{2}\|v_{i}|_{\partial U}\|^{2}\\
 & \leq Ct_{*}\sum_{U\in\mathfrak{U}}\sum_{\text{Supp}(v_{i})\subset U}\alpha_{i}^{2}\frac{\|v_{i}\|^{2}}{d^{2}}\\
 & =Ct_{*}\frac{\|u_{2}\|^{2}}{d^{2}}
\end{align*}
for some fix $C>0$ by Proposition \ref{prop:Exponential_decay} and
therefore 
\[
\|X^{\leq}\|=\sup_{u_{2}\in E_{2}^{\leq},u_{3}\in E_{3}}\frac{\langle u_{3},L(\mathbb{G})u_{2}\rangle}{\|u_{3}\|\|u_{2}\|}=O\Bigl(\frac{1}{\sqrt{d}}\Bigr).
\]
\end{itemize}
We now finish the proof of Proposition \ref{prop:blockdiagonal} we
have 
\[
V^{*}L(\mathbb{G})V=\begin{pmatrix}\nu & 0 & 0 & 0\\
0 & D^{\leq} & 0 & 0\\
0 & 0 & D^{>} & X^{>}\\
0 & 0 & X^{>} & Y
\end{pmatrix}+\begin{pmatrix}0 & 0 & 0 & X_{\nu}^{*}\\
0 & 0 & 0 & X^{\leq}\\
0 & 0 & 0 & 0\\
X_{\nu} & X^{\leq} & 0 & 0
\end{pmatrix}
\]
therefore 
\begin{equation}
\sum_{\lambda\in\text{Spec}(L(\mathbb{G}))}\delta_{\lambda}\quad\text{and}\quad\delta_{\nu+\delta\nu}+\sum_{\mu\in\text{Spec}(D^{\leq})}\delta_{\mu+\epsilon_{\mu}}\label{eq:agree_1}
\end{equation}
agrees on $[0,\tau d/2-o(1)]$ with $\delta\nu\leq\|X_{\nu}\|=O(N^{-c})$
and
\begin{align*}
\epsilon_{\mu} & \leq\min\{\|X^{\leq}\|,\frac{5\|X^{\leq}\|^{2}}{|\frac{1}{2}\tau d-\mu|}\}\leq\begin{cases}
\frac{1}{\sqrt{d}} & if\:\mu\geq\tau d/4,\\
\frac{20}{\tau d^{2}} & if\:\mu\leq\tau d/4
\end{cases}=O((1+\mu)d^{-1}),
\end{align*}
where we used Lemma \ref{lem:PerturbationLemma}.\textcolor{blue}{{}
}We now use Proposition \ref{prop:Tree_to_Ball} 
\begin{equation}
\sum_{\mu\in\text{Spec}(D^{\leq})}\delta_{\mu+\epsilon_{\mu}}=\sum_{U\in\mathfrak{U}}\sum_{\mu\in\text{Spec}(L(\mathbb{G})|_{U})}\delta_{\mu+\epsilon_{\mu}}=\sum_{F\in\mathfrak{F}}\sum_{\mu\in\text{Spec}(L(F;\widehat{\partial F}))}\delta_{\mu+\epsilon_{\mu}+\epsilon_{\mu}'}\label{eq:agree_2}
\end{equation}
with $\epsilon_{\mu}'=O\bigl((\mu+1)d^{-1}\bigr)$. Combining \eqref{eq:agree_1}
and \eqref{eq:agree_2} gives the right-hand side term of Proposition
\ref{prop:blockdiagonal}. Finally because the multiplicity of the
eigenvalue $\{0\}$ of $L(\mathbb{G})$ is equal to the number of
connected component in the graph, we actually have that $\delta_{\nu+\delta\nu}=\delta_{0}$
(which can be seen as the eigenvalue associated with the giant connected
component).
\end{proof}

\section{Spectrum of weighted Laplacian on trees , Proof of Proposition \ref{prop:minimality}.}

\label{sec:Spectrum-of-Laplacian}In this section, we prove that the
line is the optimal graph to minimize the smallest non-zero eigenvalue
and then compute explicitly its spectrum.
\begin{prop}
\label{prop:LineMinimality}For any tree $T$ of size $|T|\leq t$
and $\widehat{X}\subset[[T]]$ we have 
\begin{equation}
\lambda_{1}(L(\mathbb{L}_{t},\{1\}))\leq\min(\text{Spec}(L(T,\widehat{X}))\setminus\{0\}).\label{eq:LineMin-1}
\end{equation}
\end{prop}

Note that since the number of trees of size $t$ is finite and since
increasing the size of a multiset is a positive rank-one perturbation
of the corresponding deformed Laplacian matrix, Proposition \ref{prop:LineMinimality}
tells that for any $t\in\mathbb{N}$ there exists a universal constant
$c$>0 such that $\lambda_{1}(L(\mathbb{L}_{t},\{1\}))\leq\min(\text{Spec}(L(T,\widehat{X}))\setminus\{0\})$-c,
for every tree $T$ of size $t$ and any multiset on $[[T]].$
\begin{lem}
\label{lem:LineSpectrum}The eigenvalues $L(\mathbb{L}_{t},\{1\})$
are $\lambda_{k}=2-2\cos\left(\frac{\pi}{2t+1}+\frac{2k\pi}{2t+1}\right)$
for $0\leq k<t$.
\end{lem}

\begin{proof}[Proof of Proposition \ref{prop:minimality}]
This directly follows from Proposition \ref{prop:LineMinimality}
and Lemma \ref{lem:LineSpectrum} with $t=t_{*}$.
\end{proof}
\begin{proof}[Proof of Proposition \ref{prop:LineMinimality}]
 We are interesting in the minimiser of $\min(\text{Spec}(L(T,\widehat{X}))\setminus\{0\})$
among all the tree $T$ of size $|T|\leq t$. For any $i\in T$, $\mathrm{1}_{i}\mathrm{1}_{i}^{*}$
is positive operator therefore we have that for all multiset $\widehat{X}\subset\widehat{Y}\subset[[T]]$,
$L(T,\widehat{X})\leq L(T,\widehat{Y})$ (as operator) and then 
\[
\lambda_{1}\bigl(L\bigl(T,\widehat{X}\bigr)\bigr)\leq\lambda_{1}\bigl(L\bigl(T,\widehat{Y}\bigr)\bigr)
\]
In the case $\widehat{X}=\emptyset$ we have $\lambda_{1}(L(T,\emptyset))=\lambda_{1}(L(T))=0$
and by interlacing we have 
\[
0=\lambda_{1}(L(T,\emptyset))<\lambda_{1}\bigl(L\bigl(T,\{i\}\bigr)\bigr)<\lambda_{2}\bigl(L\bigl(T,\emptyset\bigr)\bigr).
\]
If $|T|<t$ there exists a large graph $T'$ with $|T'|=t$ such that
$T\subset T'$ and again by interlacing we have 
\[
0<\lambda_{1}\bigl(L\bigl(T',\{i\}\bigr)\bigr)<\lambda_{1}\bigl(L\bigl(T,\{i\}\bigr)\bigr)
\]
Therefore the minimiser is of the form $\lambda_{1}(L(T,\{i\}))$
with $|T|=t$.

We think of $i$ as the root of $T$ and for every edge $e=(x,y)$
of $T$ we define the weight 
\[
q(e)=u(x)-u(y)
\]
and denote $\boldsymbol{q}=(q_{e})_{e\in E}\in\mathbb{R}^{t-1}.$
We have

\begin{equation}
\lambda_{1}(L(T,\{i\}))=\inf_{u\in\mathbb{R}^{t}\setminus\{0\}}\frac{\langle u,L(T,\{i\})u\rangle}{\|u\|^{2}}=\inf_{u\in\mathbb{R}^{t}\setminus\{0\}}\frac{\sum_{e\in E}q(e)^{2}+|u(i)|^{2}}{\|u\|^{2}}.\label{eq:Gamma}
\end{equation}
Let us now change to a dual approach. We have a bijection between
$\{u(i),\boldsymbol{q}\}$ and $(u(x))_{x\in T}$ given by the equation
\begin{equation}
u(x)=u(i)+\sum_{e\in P_{i\rightarrow x}(T)}q(e)\label{eq:BijectionDual}
\end{equation}
where we denoted by $P_{i\rightarrow x}(T)$ the unique path from
$i$ to $x$ in $T.$ We denote $u=\phi(T,i,u(i),\boldsymbol{q})$
the bijection of Equation \eqref{eq:BijectionDual}. Then we have
\begin{align*}
\min_{T,i}\lambda_{1}(L(T,\{i\})) & =\inf_{u(i),\boldsymbol{q}}\min_{T,i}\frac{\sum_{e}|q(e)|^{2}+|u(i)|^{2}}{\|\phi(T,i,u(i),\boldsymbol{q})\|^{2}}=\inf_{u(i),\boldsymbol{q}}\frac{\sum_{e}|q(e)|^{2}+|u(i)|^{2}}{\max_{T,i}\|\phi(T,i,u(i),\boldsymbol{q})\|^{2}}.
\end{align*}
We now show that for fixed $u(i),\boldsymbol{q}$ the norm $\|\phi(T,i,u(i),\boldsymbol{q})\|^{2}$
is maximized when $T$ is the line and $i$ an extremal point.

We sort the edges decreasingly according to their weight 
\begin{align*}
|q(e_{1})|\geq|q(e_{2})|\geq\ldots\geq|q(e_{t-1})|\geq0
\end{align*}
and the vertices increasingly according to their graph distance to
the root $i=x_{0}$ 
\[
0\leq\text{dist}(i,x_{1})\leq\text{dist}(i,x_{2})\leq\cdots\leq\text{dist}(i,x_{t-1}).
\]
We have 
\[
    \bigg\vert\sum_{e\in P_{i\rightarrow x_{k}}}q(e)\bigg\vert\leq\sum_{e\in P_{i\rightarrow x_{k}}}|q(e)|\leq\sum_{l=1}^{|P(i,x_{k})|}|q(e_{i})|\leq\sum_{l=1}^{k}|q(e_{i})|,
\]
where we use that the length of the path $P(i,x_{k})$ is at most
$k$. Therefore 
\begin{align}
\|\phi(T,i,u(i),\boldsymbol{q})\|^{2} & =\sum_{k=0}^{t-1}|u(i)+\sum_{e\in P_{i\rightarrow x_{k}}}q(e)|^{2}\leq\sum_{k=0}^{t-1}\left||u(i)|+\sum_{l=1}^{k}|q(e_{i})|\right|^{2}\label{eq:boundual}
\end{align}
Observe that the right-hand side can be reached if $T$ is a line and
$i$ is an extremal point. Moreover this solution is unique if $q(e_{i})>0$
for all $i\leq t-1$. Then we have 
\[
\max_{T,i}\|\phi(T,i,u(i),\boldsymbol{q},T)\|^{2}=\|\phi(\mathbb{L}_{t},1,u(i),\boldsymbol{q})\|^{2}
\]
and finally conclude that $\min_{T,i}\lambda_{1}(L(T,\{i\}))=\lambda_{1}(L(\mathbb{L}_{t},\{1\}))$.

Finally, we now prove that we can assume $|q(e_{t-1})|>0$. If $q(e_{t-1})=0$
then 
\[
\frac{d}{dq(e_{t-1})}\left(\sum_{e}|q(e)|^{2}+|u(i)|^{2}\right)=0,\quad\frac{d}{dq(e_{t-1})}\sum_{k=0}^{t-1}\left||u(i)|+\sum_{l=1}^{k}|q(e_{i})|\right|^{2}\neq0,
\]
therefore $\boldsymbol{q}$ is not the minimser of 
\[
\inf_{u(i),\boldsymbol{q}}\frac{\sum_{e}|q(e)|^{2}+|u(i)|^{2}}{\|\phi(\mathbb{L}_{t},1,u(i),\boldsymbol{q})\|^{2}}.
\]
\end{proof}
\begin{proof}
[Proof of Lemma \ref{lem:LineSpectrum}]We write the matrix 
\[
L(\mathbb{L}_{t},\{1\})=\begin{pmatrix}2 & -1 & 0 & \cdots & 0 & 0\\
-1 & 2 & -1 & 0 &  & 0\\
0 & -1 & \ddots & \ddots &  & \vdots\\
\vdots & 0 & \ddots & \ddots & \ddots & 0\\
0 &  &  & \ddots & 2 & -1\\
0 & 0 & \cdots & 0 & -1 & 1
\end{pmatrix}\in\mathbb{R}^{t\times t}.
\]
Let $\lambda$ an eigenvalue and and $u\in\mathbb{R}^{t+1}$ the associated
eigenvector with the convention $u_{0}=0$. We have 
\[
Lu=\lambda u\Leftrightarrow\begin{cases}
u_{k+1}+u_{k-1}=(2-\lambda)u_{k} & \text{for }1\leq k<t\\
u_{t-1}=(1-\lambda)u_{t}
\end{cases}
\]
Therefore for any $1\leq k<t$ 
\[
\begin{pmatrix}u_{k+1}\\
u_{k}
\end{pmatrix}=P\begin{pmatrix}u_{k}\\
u_{k-1}
\end{pmatrix}=P^{k}\begin{pmatrix}u_{1}\\
0
\end{pmatrix}\quad\text{with}\quad P=\begin{pmatrix}(2-\lambda) & -1\\
1 & 0
\end{pmatrix}.
\]
Let $\alpha\in[0,\frac{\pi}{2}]$ such that $2-\lambda=2\cos\alpha$
then the eigenvalues of $P$ are solution of $\gamma^{2}-2\cos(\alpha)\gamma+1=0$
which are 
\[
\gamma_{\pm}=\frac{2\cos(\alpha)\pm i\sqrt{4-4\cos(\alpha)^{2}}}{2}=\cos\alpha\pm i\sin\alpha=e^{\pm i\alpha}.
\]
Then $u_{k}=ae^{ik\alpha}+be^{-ik\alpha}$ for some $a,b\in\mathbb{C}$
and with $u_{0}=0$ we conclude that $u_{k}=\sin(k\alpha)$ up to
a multiplicative factor. Finally, the last equation becomes 
\begin{align*}
(2\cos\alpha-1)\sin(t\alpha) & =\sin((t-1)\alpha)
\end{align*}
and using the recursive algorithm for finding the $n$th multiple
angle, we see that 
\[
\sin((t+1)\alpha)=2\cos(\alpha)\sin(t\alpha)-\sin((t-1)\alpha)=\sin(t\alpha).
\]
We conclude that $\sin((t+1)\alpha)=\sin(t\alpha)$ and thus 
\[
\alpha=\frac{\pi}{2t+1}\mod\Bigl(\frac{2\pi}{2t+1}\Bigr).
\]
\end{proof}

\section{Outlook}

\subsection{Fluctuations and eigenvector localization}

With Theorem \ref{thm:1} we understand the distribution of the smallest
eigenvalues of $L$ but at a deterministic level. A natural development
would be to improve such a prediction to obtain the fluctuation scale
and its asymptotic law. Also Theorem \ref{thm:1} does not give
any information about the corresponding eigenvectors but the proof
hints at the fact that there exists a narrow relation between the
smallest eigenvalues of $L$ and some specific geometric objects that
can be found in $G$, namely trees of a particular shape. We believe
that this relation can be made into a rigorous mathematical result
by showing that the eigenvectors corresponding to the smallest eigenvalues
are localized on the lines of maximal size. The work to get there
is non-trivial but a very similar result has been done in previous
work (\cite[Theorem 1.7]{alt2021extremal} as well as in \cite{rivier2023thesis}. Inspired by these works we expect the following to hold
:
\begin{conjecture}
There is an interval $[a,b]$, with $a,b$ close to $2-2\cos\left(\frac{\pi}{2t_{*}+1}\right)$
such that
\begin{enumerate}
\item $\lambda_{2}\in[a,b]$ with high probabily.
\item (Fluctuations) $\sum_{\lambda\in\text{Spec}(L(\mathbb{G}))}\delta_{\lambda}$
converge to a Poisson Point Process on $[a,b]$.
\item (Localization) For any $\lambda\in[a,b]$ and $v$ the corresponding
eigenvector, there exists a tree $T\subset\mathbb{G}$ that is isomorphic
to line of size $t_{*}$ and such that $\|v|_{T}\|=1-o(1)$
\end{enumerate}
\end{conjecture}

Following the strategy in \cite{alt2021poisson} one have to improve
the bijection of Proposition \ref{prop:blockdiagonal} such that the
error terms $\epsilon_{\mu}$ are much smaller. We believe that this
could be done analyzing the Laplacian on a neighborhood $U$ of $F\subset U$
instead of $L(F;\widehat{\partial F})$. For the first orders one
should replace the computation of the spectrum of $L(\mathbb{L}_{t},\{1\})$
by the one of the following tridiagonal matrix
\begin{equation}\label{eq:matrixSecondOrder}
M:= M(t_*, D_z, \vert S_2(z)\vert ) =\begin{pmatrix}1 & -1 & 0 & \cdots &  & 0 & 0\\
-1 & 2 & -1 & 0 &  &  & 0\\
0 & -1 & \ddots & \ddots &  &  & \vdots\\
\vdots & 0 & \ddots & \ddots\\
 &  &  &  & 2 & -1 & 0\\
0 &  &  &  & -1 & D_{z} & \sqrt{D_{z}}\\
0 & 0 & \cdots &  & 0 & \sqrt{D_{z}} & \frac{|S_{2}(z)|}{D_{z}}
\end{pmatrix}\in\mathbb{R}^{t_{*}+2\times t_{*}+2}.
\end{equation}
We claim that there exists an eigenvalue of $L(\mathbb{G})$ such
that $\lambda=\lambda_{1}(M)+O(d^{-2})$ and that we can obtain

\begin{equation}
\lambda_{1}(M)=2-2\cos\Bigl(\frac{\pi}{2t_{*}+1}\Bigr)+\delta\bigl(D_{z},|S_{2}(z)|\bigr)+O\bigl(d^{-2}\bigr),\label{eq:Improve_line_spectrum}
\end{equation}
for some explicit function $\delta:\mathbb{R}^{2}\rightarrow\mathbb{R}$.

\begin{figure}[ht!]
    \begin{center}
        \def\svgwidth{0.9\columnwidth} 
            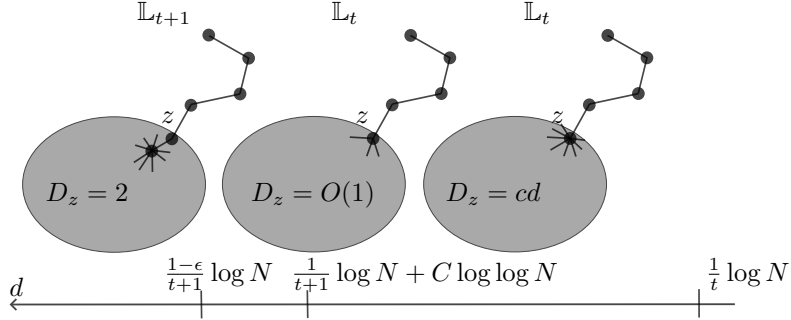
        \caption{Illustration of the intermediate regime. We expect lines of size $t_*+1$  to appear as a result of the \textit{erosion} of the anchor of lines of size $t_*$.}
        \label{fig:erosion}
    \end{center}
\end{figure}

We should also improve Proposition \ref{prop:XiProbSet} having not
only the existence of the line but controlling the degree of the
anchor as well: $\{T\sim\mathbb{L}_{t},z\text{ anchor with }D_{z}=\alpha d\}$
for some $\tau\leq\alpha<1$. Because of \eqref{eq:Improve_line_spectrum},
the smallest eigenvalue would be obtained with the smallest possible
degree of the anchor. We expect it to be close to $cd$ where
\begin{equation}\label{eq:smallest_degree_anchor}
N^{1-t^{*}d+o(1)}\times\mathbb{P}(D_{z}=cd)\asymp1.
\end{equation} 
We then would choose the interval $[a,b]$ as a small neighborhood of
$2-2\cos(\frac{\pi}{2t_{*}+1})+\delta(cd,cd^{2})$.
Finally, the main ingredient for the localization is an estimate on $|\lambda_{i}-\lambda_{j}|\gg\max\{\epsilon_{i},\epsilon_{j}\}$ combined with Lemma \ref{lem:PerturbationLemma}. An upper bound
on $\max\{\epsilon_{i},\epsilon_{j}\}$ should come from the improve
rigidity result above while a lower bound on $|\lambda_{i}-\lambda_{j}|$
should be a consequence of the Poisson point process statistic and
its limiting law.

\subsection{Intermediate regime}

We now add a few remarks about the intermediate regime 
\[
(1-\epsilon)\frac{\log N}{t_{*}}\leq d\leq(1+\epsilon)\frac{\log N}{t_{*}}.
\]
In our paper the condition \eqref{equ:regimes} is needed for two
properties : the existence of a line of size $t_{*}$ and to make
sure that all the vertices in its neighborhood have large degree (for
a spectral gap and to do perturbations theory). In the intermediate
regime, the existence of component of size $t_{*}$ becomes random
but we still have with high probability a line of size $t_{*}-1$
and no connected component of size $t_{*}+1$. Therefore one could
expected that $\lambda_{2}$ is random but because of Propositions
\ref{prop:XiProbSet}, \ref{prop:blockdiagonal}, \ref{prop:minimality}
we can still state the following.
\begin{prop}
In the intermediate regime with high probability
\[
2-2\cos\left(\frac{\pi}{2t_{*}+1}\right)+O\bigl(d^{-1}\bigr)\leq\lambda_{2}\leq2-2\cos\left(\frac{\pi}{2t_{*}-1}\right)+O\bigl(d^{-1}\bigr)
\]
\end{prop}

To refine the analysis, we expect the degree of the anchor $D_{z}$
to play a major role in the transition from the $t_{*}$ line regime
to the $t_{*}+1$ line regime, going from $D_{z}\geq\tau d$ for $d\approx(1+\epsilon)\frac{\log N}{t_{*}+1}$
and decreasing to $D_{z}\sim1$ for $d\approx\frac{\log N}{t_{*}+1}$
therefore creating a $t_{*}+1$ line. One could then try to use \eqref{eq:Improve_line_spectrum}
to compute $\lambda_{2}$ in the intermediate regime. However, because
the predictions obtained from the spectrum of finite trees only give
discrete values, there should exist regimes of $d$ for which $\lambda_{2}$
is random with fluctuations $\asymp1$. An illustration of our heuristics is given in Figure \ref{fig:erosion}.

\subsection{Numerical simulations}

\begin{figure}[ht!]
    \begin{center}
        \def\svgwidth{0.9\columnwidth} 
\begingroup%
  \makeatletter%
  \providecommand\color[2][]{%
    \errmessage{(Inkscape) Color is used for the text in Inkscape, but the package 'color.sty' is not loaded}%
    \renewcommand\color[2][]{}%
  }%
  \providecommand\transparent[1]{%
    \errmessage{(Inkscape) Transparency is used (non-zero) for the text in Inkscape, but the package 'transparent.sty' is not loaded}%
    \renewcommand\transparent[1]{}%
  }%
  \providecommand\rotatebox[2]{#2}%
  \newcommand*\fsize{\dimexpr\f@size pt\relax}%
  \newcommand*\lineheight[1]{\fontsize{\fsize}{#1\fsize}\selectfont}%
  \ifx\svgwidth\undefined%
    \setlength{\unitlength}{475.08654497bp}%
    \ifx\svgscale\undefined%
      \relax%
    \else%
      \setlength{\unitlength}{\unitlength * \real{\svgscale}}%
    \fi%
  \else%
    \setlength{\unitlength}{\svgwidth}%
  \fi%
  \global\let\svgwidth\undefined%
  \global\let\svgscale\undefined%
  \makeatother%
  \begin{picture}(1,0.86932129)%
    \lineheight{1}%
    \setlength\tabcolsep{0pt}%
    \put(0,0){\includegraphics[width=\unitlength,page=1]{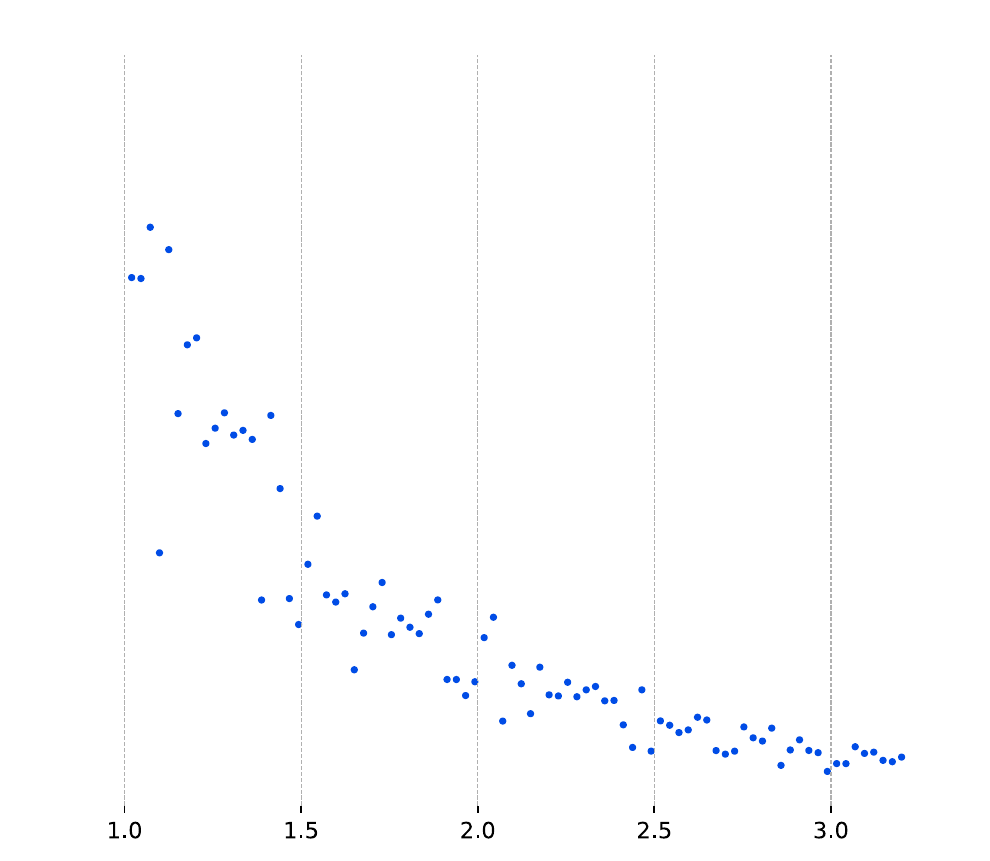}}%
    \put(0.52878316,0.00309363){\color[rgb]{0,0,0}\makebox(0,0)[lt]{\lineheight{1.14999998}\smash{\begin{tabular}[t]{l}$t$\end{tabular}}}}%
    \put(0,0){\includegraphics[width=\unitlength,page=2]{simulation_subcritical.pdf}}%
    \put(0.15973298,0.84351292){\color[rgb]{0,0,0}\makebox(0,0)[lt]{\lineheight{1.14999998}\smash{\begin{tabular}[t]{l}Sub-critical regime $d\leq\log N$, $t\geq 1$\end{tabular}}}}%
    \put(0.02922075,0.38247609){\color[rgb]{0,0,0}\rotatebox{90}{\makebox(0,0)[lt]{\lineheight{1.14999998}\smash{\begin{tabular}[t]{l}$\lambda_2(L)$\end{tabular}}}}}%
  \end{picture}%
\endgroup%

        \caption{We plotted the spectral gap of $L$ for various regimes of $d$ below the criticality threshold. Simulations were obtained with $N=10^4$, $d= \frac{1}{t} \log N$ for values of $1\leq t\leq 3.2.$ }
        \label{fig:simulations_sub}
    \end{center}
\end{figure}

\begin{figure}[ht!]
    \begin{center}
        \def\svgwidth{0.9\columnwidth} 
\begingroup%
  \makeatletter%
  \providecommand\color[2][]{%
    \errmessage{(Inkscape) Color is used for the text in Inkscape, but the package 'color.sty' is not loaded}%
    \renewcommand\color[2][]{}%
  }%
  \providecommand\transparent[1]{%
    \errmessage{(Inkscape) Transparency is used (non-zero) for the text in Inkscape, but the package 'transparent.sty' is not loaded}%
    \renewcommand\transparent[1]{}%
  }%
  \providecommand\rotatebox[2]{#2}%
  \newcommand*\fsize{\dimexpr\f@size pt\relax}%
  \newcommand*\lineheight[1]{\fontsize{\fsize}{#1\fsize}\selectfont}%
  \ifx\svgwidth\undefined%
    \setlength{\unitlength}{313.99222066bp}%
    \ifx\svgscale\undefined%
      \relax%
    \else%
      \setlength{\unitlength}{\unitlength * \real{\svgscale}}%
    \fi%
  \else%
    \setlength{\unitlength}{\svgwidth}%
  \fi%
  \global\let\svgwidth\undefined%
  \global\let\svgscale\undefined%
  \makeatother%
  \begin{picture}(1,0.80898584)%
    \lineheight{1}%
    \setlength\tabcolsep{0pt}%
    \put(0,0){\includegraphics[width=\unitlength,page=1]{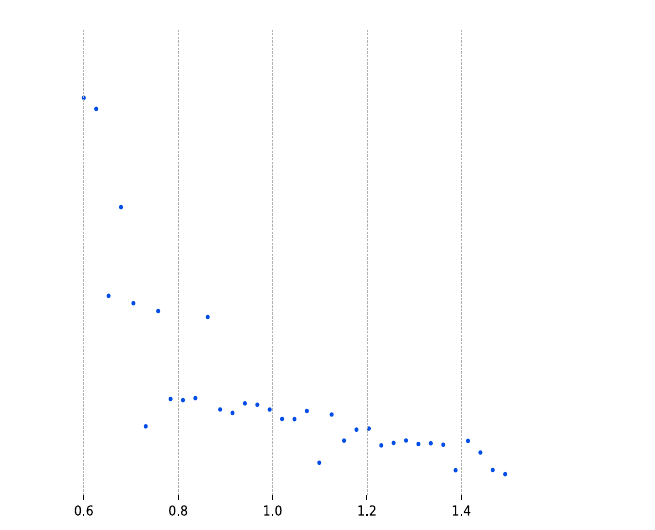}}%
    \put(0.44400771,0.00468077){\color[rgb]{0,0,0}\makebox(0,0)[lt]{\lineheight{1.14999998}\smash{\begin{tabular}[t]{l}$t$\end{tabular}}}}%
    \put(0,0){\includegraphics[width=\unitlength,page=2]{simulations_image.pdf}}%
    \put(0.0241969,0.39163584){\color[rgb]{0,0,0}\rotatebox{90}{\makebox(0,0)[lt]{\lineheight{1.14999998}\smash{\begin{tabular}[t]{l}$\lambda_2(L)$\end{tabular}}}}}%
    \put(0,0){\includegraphics[width=\unitlength,page=3]{simulations_image.pdf}}%
    \put(0.13743085,0.78478886){\color[rgb]{0,0,0}\makebox(0,0)[lt]{\lineheight{1.14999998}\smash{\begin{tabular}[t]{l}Critical regime $d\geq \frac{1}{2}\log N$, $t\leq 1.5$\end{tabular}}}}%
  \end{picture}%
\endgroup%

        \caption{We plotted the spectral gap of $L$ for $d\geq \frac{1}{2}\log N$. Simulations were obtained with $N=10^4$, $d= \frac{1}{t} \log N$ for values of $0.6\leq t\leq 1.$  }
        \label{fig:simulations_sur}
    \end{center}
\end{figure}

In this section, we compare our estimates from Theorem \ref{thm:1} with numerical simulations presented in Figure \ref{fig:simulations_sub} and \ref{fig:simulations_sur}. Using Theorem \ref{thm:1}, we obtain a first-order approximation (dark red): note that this estimate only depends on $\lfloor t \rfloor$ (see \eqref{eq:master-equation}) and is thus a piece-wise constant function.
The error between the dark red line and the blue dots remains of the order $O\bigl(d^{-1}\bigr)$, in agreement with \eqref{eq:master-equation} (indeed for $N=10^4$, $1/d \sim 0.1$).
We expect the fit to become much better as $d\rightarrow +\infty$, but unfortunately, since $d \asymp \log N$, we would need to simulate exponentially large graphs to reduce the error, which is infeasible in practice.
Second, we can use the discussion at the beginning of the second to approximate $\lambda_2(L)$ by $\lambda_1(M)$ for $M$ defined in \eqref{eq:matrixSecondOrder}. 
The matrix $M$ has three parameters namely $t_*$ which corresponds to $\lfloor t \rfloor$, $D_z$ and $\vert S_2(z)\vert.$ We set $D_z=cd$, $c>0$,  to be the (expected) smallest degree of any anchor of a line of size $t_*$ and $\vert S_2(z)\vert=dD_z$ (see the discussion leading to \eqref{eq:smallest_degree_anchor}).
The second-order approximation seems to fit the numerical results much more closely, than the first-order approximation.  The simulations are displayed in Figure \ref{fig:simulations_sub}.

To give a more complete view of the behavior of the spectral gap, we also provide simulations for $d\geq \log N$ on Figure \ref{fig:simulations_sur} It was shown in \cite{rivier2023thesis}[Chapter 3] that the spectral gap of $L$ in these regimes was given to the first order by $\Delta + \frac{d}{\Delta - d}$, $\Delta := \min_{x\in[N]} D_x.$ Therefore, whenever $d \geq \log N$, we use $\Delta$ as a first order approximation for $\lambda_2(L)$ (dark green line) and $\Delta + \frac{d}{\Delta - d}$ as a second order approximation (light green line).  In the subcritical regime, we can approximate $\lambda_2(L)$ in two different ways.

\section{Appendix}

We mention \cite[Lemma E.1 and E.5]{alt2021poisson}.
\begin{lem}
\label{lem:adj_of_tree}For $T$ a tree with degrees bounded by $M>0$
and $A$ its adjacency matrix we have $\Vert A\Vert\leq\sqrt{2M}.$
\end{lem}

\begin{lem}
\label{lem:PerturbationLemma}Let $M$ be a self-adjoint matrix. Let
$\epsilon,\Delta>0$ satisfy $5\epsilon\leq\Delta$. Let $\lambda\in\mathbb{R}$
and suppose that $M$ has a unique eigenvalue $\mu$ in $[\lambda-\Delta,\lambda+\Delta]$,
with corresponding normalized eigenvector $w$. If there exists a
normalized vector $v$ such that $\|(M-\lambda)v\|\leq\epsilon$ then

\[
\mu-\lambda=\langle v,(M-\lambda)v\rangle+O\left(\frac{\epsilon^{2}}{\Delta}\right),\qquad\|w-v\|=O\left(\frac{\epsilon^{2}}{\Delta}\right).
\]

\end{lem}

\bibliographystyle{alpha}
\bibliography{bibli_tree}

\end{document}